\documentclass[10.5pt,a4paper]{amsart}
\usepackage{etex} 
\usepackage[latin1]{inputenc}
\usepackage[T1]{fontenc}
\usepackage[right=3cm,left=3cm,top=2cm,bottom=2cm]{geometry}
\usepackage{ifthen}
\usepackage{enumerate}
\usepackage{amsmath}
\usepackage{amsthm}
\usepackage{amsfonts}
\usepackage{amssymb}
\usepackage{latexsym}
\usepackage{ae}
\usepackage{xcolor} 
\usepackage{graphics,amsmath,amssymb}
\usepackage[np]{numprint}
\npdecimalsign{\ensuremath{.}}

\title{When the number of divisors is a quadratic residue}
\author{Olivier Bordell\`{e}s}
\address{2 all\'{e}e de la combe \\ 43000 Aiguilhe \\ France}
\email{borde43@wanadoo.fr}
\date{}
\dedicatory{}

\theoremstyle{plain}
\newtheorem{theorem}{Theorem}

\newtheorem{lemma}[theorem]{Lemma}
\newtheorem{proposition}[theorem]{Proposition}

\theoremstyle{definition}

\newtheorem{example}[theorem]{Example}

\theoremstyle{remark}
\newtheorem{remark}[theorem]{Remark}

\newcommand{\Z}{\mathbb {Z}}
\newcommand{\Q}{\mathbb {Q}}
\newcommand{\R}{\mathbb {R}}
\newcommand{\C}{\mathbb {C}}

\DeclareMathOperator{\md}{mod}

\makeatletter
\@namedef{subjclassname@2010}{%
  \textup{2010} Mathematics Subject Classification}
\makeatother

\begin{document}

\begin{abstract}
Let $q > 2$ be a prime number and define $\lambda_q := \left( \frac{\tau}{q} \right)$ where $\tau(n)$ is the number of divisors of $n$ and $\left( \frac{\cdot}{q} \right)$ is the Legendre symbol. When $\tau(n)$ is a quadratic residue modulo $q$, then $\left( \lambda_q \star \mathbf{1} \right) (n)$ could be close to the number of divisors of $n$. This is the aim of this work to compare the mean value of the function $\lambda_q \star \mathbf{1}$ to the well known average order of $\tau$. The proof reveals that the results depend heavily on the value of $\left( \frac{2}{q} \right)$. A bound for short sums in the case $q=5$ is also given, using profound results from the theory of integer points close to certain smooth curves.
\end{abstract}


\maketitle

\thispagestyle{myheadings}
\font\rms=cmr8 
\font\its=cmti8 
\font\bfs=cmbx8

\section{Introduction and main result}

If $\lambda = (-1)^{\Omega}$ is the Liouville function, then 
$$L(s,\lambda) = \frac{\zeta(2s)}{\zeta(s)} \quad \left( \sigma > 1 \right).$$
This implies the convolution identity
$$\sum_{n \leqslant x} \left( \lambda \star \mathbf{1} \right) (n) = \left \lfloor x^{1/2} \right \rfloor.$$

Define $\lambda_3 := \left( \frac{\tau}{3} \right)$ where where $\tau(n)$ is the number of divisors of $n$ and $\left( \frac{\cdot}{3} \right)$ is the Legendre symbol modulo $3$. Then from Proposition~\ref{pro2} below
$$L(s,\lambda_3) = \frac{\zeta(3s)}{\zeta(s)} \quad \left( \sigma > 1 \right)$$
implying the convolution identity
$$\sum_{n \leqslant x} \left( \lambda_3 \star \mathbf{1} \right) (n) = \left \lfloor x^{1/3} \right \rfloor.$$

Now let $q >2$ be a prime number and define $\lambda_q := \left( \frac{\tau}{q} \right)$ where $\left( \frac{\cdot}{q} \right)$ is the Legendre symbol modulo $q$. Our main aim is to investigate the sum
$$\sum_{n \leqslant x} \left( \lambda_q \star \mathbf{1} \right) (n).$$
When $\tau(n)$ is a quadratic residue modulo $q$, one may wonder if $\left( \lambda_q \star \mathbf{1} \right) (n)$ has a high probability to be equal to the number of divisors of $n$. It then could be interesting to study its average order and to compare it to that of $\tau$, i.e.
\begin{equation}
   \sum_{n \leqslant x} \tau(n) = x \left( \log x + 2 \gamma - 1 \right) + O \left( x^{\theta + \varepsilon} \right) \label{e1}
\end{equation}
where $\frac{1}{4} \leqslant \theta \leqslant \frac{131}{416}$, the left-hand side being established by Hardy \cite{har}, the right-hand side being the best estimate to date due to Huxley \cite{hux}. The main result of this paper can be stated as follows.

\begin{theorem}
\label{t1}
Let $q >3$ be a prime number.
\begin{enumerate}
  \item[\scriptsize $\triangleright$] If $q \equiv \pm 1 \; (\md 8)$
$$\sum_{n \leqslant x} \left( \lambda_q \star \mathbf{1} \right) (n) = x \zeta(q) P_q(1) \left\lbrace \log x + 2 \gamma - 1 + q \frac{\zeta^\prime}{\zeta}(q) + \frac{P_q^\prime}{P_q} (1) \right\rbrace + O_{q,\varepsilon} \left( x^{\max \left( 1/c_q, \theta \right) + \varepsilon} \right)$$
where $\theta$ is defined in \eqref{e1}, $c_q$ is given in \eqref{e2} and
\begin{eqnarray*}
   & & P_q(1) = \prod_p \left( 1 + \sum_{m=c_q}^{q-1} \left\lbrace \left( \frac{m+1}{q} \right) - \left( \frac{m}{q} \right) \right\rbrace \frac{1}{p^{m}} \right) \\
   & & \frac{P_q^\prime}{P_q} (1) = - \sum_p \log p \left( \frac{\sum_{m=c_q}^{q-1} \left\lbrace \left( \frac{m+1}{q} \right) - \left( \frac{m}{q} \right) \right\rbrace \frac{m}{p^{m}}}{1+\sum_{m=c_q}^{q-1} \left\lbrace \left( \frac{m+1}{q} \right) - \left( \frac{m}{q} \right) \right\rbrace \frac{1}{p^{m}}} \right).
\end{eqnarray*}
   \item[\scriptsize $\triangleright$] If $q \equiv \pm 11 \; (\md 24)$
$$\sum_{n \leqslant x} \left( \lambda_q \star \mathbf{1} \right) (n) = x^{1/2} \zeta\left( \tfrac{q}{2} \right) R_q \left( \tfrac{1}{2} \right) + O_{q,\varepsilon} \left( x^{1/3 + \varepsilon} \right)$$
where
$$R_q \left( \tfrac{1}{2} \right):=\prod_p \left( 1 + \sum_{m=3}^{q-1} \left\lbrace \left( \frac{m+1}{q} \right) + \left( \frac{m}{q} \right) \right\rbrace \frac{1}{p^{m/2}} \right).$$
   \item[\scriptsize $\triangleright$] If $q \equiv \pm 5 \; (\md 24)$, there exists $c > 0$ such that
   $$\sum_{n \leqslant x} \left( \lambda_q \star \mathbf{1} \right) (n) \ll_q x^{1/2} e^{-c \left( \log x^{1/4}\right)^{3/5} \left( \log \log x^{1/4} \right)^{-1/5}}.$$
Furthermore, if the Riemann hypothesis is true, then for $x$ sufficiently large
$$\sum_{n \leqslant x} \left( \lambda_q \star \mathbf{1} \right) (n) \ll_{q,\varepsilon} x^{1/4} e^{\left( \log \sqrt x \right)^{1/2} (\log \log \sqrt x)^{5/2+\varepsilon}}.$$
\end{enumerate}
\end{theorem}

\begin{example}
\label{ex}
\begin{eqnarray*}
   & & \sum_{n \leqslant x} \left( \lambda_7 \star \mathbf{1} \right) (n) \doteq \np{0.454} \, x \left( \log x + 2 \gamma + \np{0.784} \right) + O_\varepsilon \left( x^{1/2 + \varepsilon} \right). \\  
   & & \sum_{n \leqslant x} \left( \lambda_{23} \star \mathbf{1} \right) (n) \doteq \np{0.899} \, x \left( \log x + 2 \gamma - \np{0.678} \right) + O_\varepsilon \left( x^{131/416 + \varepsilon} \right). \\  
   & & \sum_{n \leqslant x} \left( \lambda_{13} \star \mathbf{1} \right) (n) \doteq \np{1.969} \, x^{1/2} + O_\varepsilon \left( x^{1/3 + \varepsilon} \right). \\  
   & & \sum_{n \leqslant x} \left( \lambda_5 \star \mathbf{1} \right) (n) \ll x^{1/2} e^{-c \left( \log x^{1/4}\right)^{3/5} \left( \log \log x^{1/4} \right)^{-1/5}}.  
\end{eqnarray*}
\end{example}

\section{Notation}

In what follows, $x \geqslant e^{4}$ is a large real number, $\varepsilon \in \left( 0, \frac{1}{4} \right)$ is a small real number which does not need to be the same at each occurrence, $s := \sigma + it \in \C$, $q$ always denotes an odd prime number, $\left( \frac{\cdot}{q} \right)$ is the Legendre symbol modulo $q$ and define
$$\lambda_q := \left( \frac{\tau}{q} \right)$$
where $\tau(n) := \sum_{d \mid n} 1$. Also, $\mathbf{1}$ is the constant arithmetic function equal to $1$.

\medskip

For any arithmetic functions $F$ and $G$, $L(s,F)$ is the Dirichlet series of $F$, the Dirichlet convolution product $F \star G$ is defined by
$$(F \star G)(n) := \sum_{d \mid n} F(d) G(n/d)$$
and $F^{-1}$ is the Dirichlet convolution inverse of $F$. If $r \in \Z_{\geqslant 2}$, then
$$a_r(n) := \begin{cases} 1, & \textrm{if\ } n=m^r ; \\ 0, & \textrm{otherwise}. \end{cases}$$
For some $c >0$, set
$$\delta_c(x) := e^{-c (\log x)^{3/5} (\log \log x)^{-1/5}} \quad \textrm{and} \quad \omega(x) := e^{\left( \log x \right)^{1/2} (\log \log x)^{5/2+\varepsilon}}.$$

Finally, let $M(x)$ and $L(x)$ be respectively the Mertens function and the summatory function of the Liouville function, i.e.
$$M(x) := \sum_{n \leqslant x} \mu(n) \quad \textrm{and} \quad L(x) := \sum_{n \leqslant x} \lambda(n).$$

\section{The Dirichlet series of $\lambda_q$}

\begin{proposition}
\label{pro2}
Let $q \geqslant 3$ be a prime number. For any $s \in \C$ such that $\sigma > 1$
\begin{enumerate}
   \item[\scriptsize $\triangleright$] If $q \equiv \pm 1 \; (\md 8)$
$$L \left( s,\lambda_q \right) = \zeta(qs) \zeta(s) \prod_p \left( 1 + \sum_{m=c_q}^{q-1} \left\lbrace \left( \frac{m+1}{q} \right) - \left( \frac{m}{q} \right) \right\rbrace \frac{1}{p^{ms}} \right)$$
where
\begin{equation}
   c_q := \begin{cases} 2, & \mathrm{if\ } q \equiv \pm 7 \; (\md 24) \, ; \\ \geqslant 4, & \mathrm{if\ } q \equiv \pm 1 \; (\md 24). \end{cases} \label{e2}
\end{equation}
   \item[\scriptsize $\triangleright$] If $q \equiv \pm 3 \; (\md 8)$
$$L \left( s,\lambda_q \right) = \frac{\zeta(qs) \zeta(2s)}{\zeta(s)} \prod_p \left( 1  + \sum_{m=d_q}^{q-1} \left\lbrace \left( \frac{m+1}{q} \right) + \left( \frac{m}{q} \right) \right\rbrace  \frac{1}{p^{ms}} \right)$$
where
$$d_q := \begin{cases} 2, & \mathrm{if\ } q \equiv \pm 5 \; (\md 24) \ \mathrm{or\ } q = 3 \, ; \\ 3, & \mathrm{if\ } q \equiv \pm 11 \; (\md 24). \end{cases}$$
\end{enumerate}
\end{proposition}

\begin{proof}
Set $\chi_q := \left( \frac{\cdot}{q} \right)$ for convenience. From \cite[Lemma 2.1]{mut}, we have 
\begin{eqnarray*}
   L \left( s,\lambda_q \right) &=& \prod_p \left( 1 + \sum_{\alpha = 1}^\infty \frac{\chi_q ( \alpha + 1)}{p^{s \alpha}} \right) = \prod_p \left( 1 + p^s \sum_{\alpha = 2}^\infty \frac{\chi_q ( \alpha )}{p^{s \alpha}} \right) \\
   &=& \prod_p \left \{ 1 + p^s \left( \left( 1 - \frac{1}{p^{qs}} \right)^{-1} \sum_{m=1}^{q-1} \left( \frac{m}{q} \right) \frac{1}{p^{ms}} - p^{-s} \right) \right \} \\
   &=& \prod_p  \left\{ \left( 1 - \frac{1}{p^{qs}} \right)^{-1} \sum_{m=1}^{q-1} \left( \frac{m}{q} \right) \frac{1}{p^{(m-1)s}} \right\rbrace \\
   &=& \zeta(qs) \prod_p \left( 1 + \sum_{m=2}^{q-1} \left( \frac{m}{q} \right) \frac{1}{p^{(m-1)s}} \right).
\end{eqnarray*}
If $q \equiv \pm 1 \; (\md 8)$, then $\left( \frac{2}{q} \right) = 1$ and
$$L \left( s,\lambda_q \right) = \zeta(qs) \zeta(s) \prod_p \left( 1 - \frac{1}{p^s} + \left( 1 - \frac{1}{p^s} \right) \sum_{m=2}^{q-1} \left( \frac{m}{q} \right) \frac{1}{p^{(m-1)s}} \right)$$
where
\begin{eqnarray*}
   \left( 1 - \frac{1}{p^s} \right) \sum_{m=2}^{q-1} \left( \frac{m}{q} \right) \frac{1}{p^{(m-1)s}} &=& \sum_{m=2}^{q-1} \left( \frac{m}{q} \right) \left( \frac{1}{p^{(m-1)s}} - \frac{1}{p^{ms}} \right) \\
   &=& \sum_{m=1}^{q-2} \left( \frac{m+1}{q} \right) \frac{1}{p^{ms}} - \sum_{m=2}^{q-1} \left( \frac{m}{q} \right) \frac{1}{p^{ms}} \\
   &=& \left( \frac{2}{q} \right) \frac{1}{p^s} + \sum_{m=2}^{q-1} \left\lbrace \left( \frac{m+1}{q} \right) - \left( \frac{m}{q} \right) \right\rbrace  \frac{1}{p^{ms}} - \left( \frac{q}{q} \right) \frac{1}{p^{(q-1)s}} \\
   &=& \sum_{m=2}^{q-1} \left\lbrace \left( \frac{m+1}{q} \right) - \left( \frac{m}{q} \right) \right\rbrace \frac{1}{p^{ms}} + \frac{1}{p^s}.
\end{eqnarray*}
Similarly, if $q \equiv \pm 3 \; (\md 8)$, then $\left( \frac{2}{q} \right) = -1$ and
$$L \left( s,\lambda_q \right) = \frac{\zeta(qs) \zeta(2s)}{\zeta(s)} \prod_p \left( 1 + \frac{1}{p^s} + \left( 1 + \frac{1}{p^s} \right) \sum_{m=2}^{q-1} \left( \frac{m}{q} \right) \frac{1}{p^{(m-1)s}} \right)$$
where
\begin{eqnarray*}
   \left( 1 + \frac{1}{p^s} \right) \sum_{m=2}^{q-1} \left( \frac{m}{q} \right) \frac{1}{p^{(m-1)s}} &=& \sum_{m=2}^{q-1} \left( \frac{m}{q} \right) \left( \frac{1}{p^{(m-1)s}} + \frac{1}{p^{ms}} \right) \\
   &=& \sum_{m=1}^{q-2} \left( \frac{m+1}{q} \right) \frac{1}{p^{ms}} + \sum_{m=2}^{q-1} \left( \frac{m}{q} \right) \frac{1}{p^{ms}} \\
   &=& \left( \frac{2}{q} \right) \frac{1}{p^s} + \sum_{m=2}^{q-1} \left\lbrace \left( \frac{m+1}{q} \right) + \left( \frac{m}{q} \right) \right\rbrace  \frac{1}{p^{ms}} - \left( \frac{q}{q} \right) \frac{1}{p^{(q-1)s}} \\
   &=& \sum_{m=2}^{q-1} \left\lbrace \left( \frac{m+1}{q} \right) + \left( \frac{m}{q} \right) \right\rbrace \frac{1}{p^{ms}} - \frac{1}{p^s}.
\end{eqnarray*}
We achieve the proof noting that, if $q \equiv \pm 1 \; (\md 24)$, then $\left( \frac{3}{q} \right) - \left( \frac{2}{q} \right)=\left( \frac{4}{q} \right) - \left( \frac{3}{q} \right) = 0$ and, similarly, if $q \equiv \pm 11 \; (\md 24)$, then $\left( \frac{3}{q} \right) + \left( \frac{2}{q} \right) = 0$ whereas $\left( \frac{4}{q} \right) + \left( \frac{3}{q} \right)=2$.
\end{proof}

\section{Proof of Theorem~\ref{t1}}

\subsection{The case $q \equiv \pm 1 \; (\md 8)$}

For $\sigma > 1$, we set
$$G_q(s) = \zeta(qs) \prod_p \left( 1 + \sum_{m=c_q}^{q-1} \left\lbrace \left( \frac{m+1}{q} \right) - \left( \frac{m}{q} \right) \right\rbrace \frac{1}{p^{ms}} \right) := \zeta(qs) P_q(s) := \sum_{n=1}^\infty \frac{g_q(n)}{n^s}.$$
First observe that $c_q < q$ in the case $q \equiv \pm 1 \; (\md 24)$. Indeed, among the $q-4$ integers $m \in \{4, \dotsc,q-1 \}$, it is known from \cite[p.76]{dav} that there are $\frac{1}{2} (q-3) - 3$ of them such that $\left( \frac{m}{q} \right) = \left( \frac{m+1}{q} \right)$. Consequently there are $\frac{1}{2}(q+1)$ integers $m \in \{4, \dotsc,q-1 \}$ verifying $\left( \frac{m}{q} \right) \neq \left( \frac{m+1}{q} \right)$, and the inequality follows.

\medskip

Thus this Dirichlet series is absolutely convergent in the half-plane $\sigma > \frac{1}{c_q}$ where $c_q$ is given in \eqref{e2}, so that
$$\sum_{n \leqslant x} |g_q(n)| \ll_{q,\varepsilon} x^{1/c_q + \varepsilon}.$$
By partial summation, we infer
\begin{eqnarray*}
   \sum_{n \leqslant x} \frac{g_q(n)}{n} &=& \zeta(q) P_q(1) + O \left( x^{-1+1/c_q + \varepsilon} \right) \\
   & & \\
   \sum_{n \leqslant x} \frac{g_q(n)}{n} \log \frac{x}{n} &=& \zeta(q) P_q(1) \log x  + q P_q(1) \zeta^\prime(q) + P_q^\prime (1) \zeta(q) + O \left( x^{-1+1/c_q + \varepsilon} \right). \\
\end{eqnarray*}
From Proposition~\ref{pro2}, $\lambda_q \star \mathbf{1} = g_q \star \tau$. Consequently
\begin{eqnarray*}
   \sum_{n \leqslant x} \left( \lambda_q \star \mathbf{1} \right) (n) &=& \sum_{d \leqslant x} g_q (d) \sum_{k \leqslant x/d} \tau(k) \\
   &=& \sum_{d \leqslant x} g_q (d) \left\lbrace \frac{x}{d} \log \frac{x}{d} + (2 \gamma - 1 ) \frac{x}{d} + O \left( \left( \frac{x}{d} \right)^{\theta + \varepsilon} \right) \right\rbrace \\
   &=& x \left\lbrace \zeta(q) P_q(1) \log x + q P_q(1) \zeta^\prime(q) + P_q^\prime (1) \zeta(q) + (2 \gamma - 1) \zeta(q) P_q(1) \right\rbrace \\
   & & {} + O \left( x^{\max \left( 1/c_q, \theta \right) + \varepsilon} \right)
\end{eqnarray*}
where $\theta$ is defined in \eqref{e1} and where we used
$$x^{- \varepsilon} \sum_{d \leqslant x} \frac{|g_q(d)|}{d^\theta} \ll \begin{cases} x^{1/c_q-\theta}, & \textrm{if\ } c_q^{-1} \geqslant \theta \, ; \\ 1, & \textrm{otherwise}. \end{cases}$$

\subsection{The case $q \equiv \pm 11 \; (\md 24)$}

For $\sigma > 1$, we set
$$H_q(s) = \zeta(qs) \prod_p \left( 1 + \sum_{m=3}^{q-1} \left\lbrace \left( \frac{m+1}{q} \right) + \left( \frac{m}{q} \right) \right\rbrace \frac{1}{p^{ms}} \right) := \zeta(qs) R_q(s) := \sum_{n=1}^\infty \frac{h_q(n)}{n^s}.$$
Since $q > 5$, this Dirichlet series is absolutely convergent in the half-plane $\sigma > \frac{1}{3}$, so that
$$\sum_{n \leqslant x} |h_q(n)| \ll_{q,\varepsilon} x^{1/3 + \varepsilon}.$$
From Proposition~\ref{pro2}, $\lambda_q \star \mathbf{1} = h_q \star a_2$, hence
\begin{eqnarray*}
   \sum_{n \leqslant x} \left( \lambda_q \star \mathbf{1} \right) (n) &=& \sum_{d \leqslant x} h_q (d) \left \lfloor \sqrt{\frac{x}{d}} \right \rfloor \\
   &=& x^{1/2} \sum_{d \leqslant x} \frac{h_q (d)}{\sqrt{d}} + O \left( x^{1/3 + \varepsilon} \right) \\
   &=& x^{1/2} H_q \left( \tfrac{1}{2} \right) + O \left( x^{1/3 + \varepsilon} \right).
\end{eqnarray*}

\subsection{The case $q \equiv \pm 5 \; (\md 24)$}

In this case, it is necessary to rewrite $L(s,\lambda_q)$ in the following shape.

\begin{lemma}
\label{le3}
Assume $q \equiv \pm 5 \; (\md 24)$. For any $\sigma > 1$, $L(s,\lambda_q) = \dfrac{K_q(s)}{\zeta(s) \zeta(2s)}$ with
$$K_q(s) := \begin{cases} \zeta(5s), & \mathrm{if\ } q = 5 \\ & \\ \zeta(4s) L_q(s), & \mathrm{if\ } q \equiv \pm 19, \pm 29 \; (\md 120) \\ & \\ \dfrac{\mathcal{L}_q(s)}{\zeta(4s)} , & \mathrm{if\ } q \equiv \pm 43, \pm 53 \; (\md 120) \end{cases}$$
where
$$L_q(s) := \zeta(qs) \prod_p \left( 1  + \frac{2 \left (p^{2s} + p^s + 1 \right )}{p^{7s} - p^{5s}} + \frac{p^{2s}+1}{p^{2s}-1} \sum_{m=6}^{q-1} \left\lbrace \left( \frac{m+1}{q} \right) + \left( \frac{m}{q} \right) \right\rbrace  \frac{1}{p^{ms}} \right)$$
and
\begin{eqnarray*}
   \mathcal{L}_q(s) &:=& \zeta(qs) \prod_p \left( 1 - \frac{2p^{2s}-1}{\left( p^{2s}-1 \right)^3 \left( p^{2s}+1 \right)} \right. \\
   & & \left.  {} + \frac{p^{8s}}{\left( p^{2s}-1 \right)^3 \left( p^{2s}+1 \right)} \sum_{m=6}^{q-1} \left\lbrace \left( \frac{m+1}{q} \right) + \left( \frac{m}{q} \right) \right\rbrace  \frac{1}{p^{ms}} \right).
\end{eqnarray*}
The Dirichlet series $L_q$ is absolutely convergent in the half-plane $\sigma > \frac{1}{5}$, and the Dirichlet series $\mathcal{L}_q$ is absolutely convergent in the half-plane $\sigma > \frac{1}{6}$.
\end{lemma}

\begin{proof}
From Proposition~\ref{pro2}, we immediately get
\begin{equation}
   L(s,\lambda_5) = \frac{\zeta(5s)}{\zeta(s) \zeta(2s)}. \label{e3}
\end{equation}
Now suppose $q > 5$ and $q \equiv \pm 5 \; (\md 24)$. In this case, $\left( \frac{3}{q} \right) + \left( \frac{2}{q} \right) = -2$ and $\left( \frac{4}{q} \right) + \left( \frac{3}{q} \right)=0$ so that we may write by Proposition~\ref{pro2}
\begin{eqnarray*}
   L \left( s,\lambda_q \right) &=& \frac{\zeta(qs) \zeta(2s)}{\zeta(s)} \prod_p \left( 1  - \frac{2}{p^{2s}}+ \sum_{m=4}^{q-1} \left\lbrace \left( \frac{m+1}{q} \right) + \left( \frac{m}{q} \right) \right\rbrace  \frac{1}{p^{ms}} \right) \\
   &=& \frac{K_q(s)}{\zeta(s) \zeta(2s)} 
\end{eqnarray*}
where
$$K_q(s) := \zeta(qs) \prod_p \left( 1  - \frac{1}{\left( p^{2s} - 1 \right)^2}+ \frac{p^{4s}}{\left( p^{2s} - 1 \right)^2} \sum_{m=4}^{q-1} \left\lbrace \left( \frac{m+1}{q} \right) + \left( \frac{m}{q} \right) \right\rbrace  \frac{1}{p^{ms}} \right).$$
Assume $q \equiv \pm 19, \pm 29 \; (\md 120)$. Then
$$\left( \frac{5}{q} \right) + \left( \frac{4}{q} \right) = \left( \frac{6}{q} \right) + \left( \frac{5}{q} \right) =2.$$
$K_q(s)$ can therefore be written as
\begin{eqnarray*}
   K_q(s) &=& \zeta(qs) \prod_p \left( 1  + \frac{p^s+2}{p^s \left( p^{2s} - 1 \right)^2}+ \frac{p^{4s}}{\left( p^{2s} - 1 \right)^2} \sum_{m=6}^{q-1} \left\lbrace \left( \frac{m+1}{q} \right) + \left( \frac{m}{q} \right) \right\rbrace  \frac{1}{p^{ms}} \right) \\
   &=& \zeta(qs) \zeta(4s) \prod_p \left( 1  + \frac{2 \left (p^{2s} + p^s + 1 \right )}{p^{7s} - p^{5s}} + \frac{p^{2s}+1}{p^{2s}-1} \sum_{m=6}^{q-1} \left\lbrace \left( \frac{m+1}{q} \right) + \left( \frac{m}{q} \right) \right\rbrace  \frac{1}{p^{ms}} \right) \\
   &=& \zeta(4s) L_q(s).
\end{eqnarray*}
Similarly, if $q \equiv \pm 43, \pm 53 \; (\md 120)$, then
$$\left( \frac{5}{q} \right) + \left( \frac{4}{q} \right) = \left( \frac{6}{q} \right) + \left( \frac{5}{q} \right) = 0.$$
Hence
\begin{eqnarray*}
   K_q(s) &:=& \zeta(qs) \prod_p \left( 1  - \frac{1}{\left( p^{2s} - 1 \right)^2}+ \frac{p^{4s}}{\left( p^{2s} - 1 \right)^2} \sum_{m=6}^{q-1} \left\lbrace \left( \frac{m+1}{q} \right) + \left( \frac{m}{q} \right) \right\rbrace  \frac{1}{p^{ms}} \right) \\
   &=& \frac{\mathcal{L}_q(s)}{\zeta(4s)}.
\end{eqnarray*}
The proof is complete.
\end{proof}

We now are in a position to prove Theorem~\ref{t1} in the case $q \equiv \pm 5 \; (\md 24)$.

\medskip

Assume first that $q \equiv \pm 19, \pm 29 \; (\md 120)$ and let $\ell_q(n)$ be the $n$-th coefficient of the Dirichlet series $L_q(s)$. From Lemma~\ref{le3}, $\lambda_q \star \mathbf{1} = \ell_q \star a_4 \star a_2^{-1}$ and therefore
$$\sum_{n \leqslant x} \left( \lambda_q \star \mathbf{1} \right) (n) = \sum_{d \leqslant x} \ell_q(d) \sum_{m \leqslant (x/d)^{1/4}} M \left( \frac{1}{m^2}\sqrt{\frac{x}{d}} \right) = \sum_{d \leqslant x} \ell_q(d) L \left( \sqrt{\frac{x}{d}} \right).$$
Since $L(z) \ll z \delta_c(z)$ for some $c >0$
\begin{eqnarray*}
   \sum_{n \leqslant x} \left( \lambda_q \star \mathbf{1} \right) (n) & \ll & x^{1/2} \sum_{d \leqslant x} \frac{|\ell_q(d)|}{\sqrt{d}} \, \delta_c \left( \sqrt{\frac{x}{d}} \right) \\
   & \ll & x^{1/2} \left( \sum_{d \leqslant \sqrt{x}} + \sum_{\sqrt{x} < d \leqslant x} \right) \frac{|\ell_q(d)|}{\sqrt{d}} \, \delta_c \left( \sqrt{\frac{x}{d}} \right) \\
   & \ll & x^{1/2} \delta_c \left( x^{1/4} \right) + x^{1/2} \sum_{d > \sqrt{x}} \frac{|\ell_q(d)|}{\sqrt{d}}.
\end{eqnarray*}
The Dirichlet series $L_q(s) := \sum_{n=1}^\infty \ell_q(n) n^{-s}$ is absolutely convergent in the half-plane $\sigma > \frac{1}{5}$, consequently
$$\sum_{d \leqslant z} \left | \ell_q(d) \right | \ll_{q,\varepsilon} z^{1/5 + \varepsilon}$$
and by partial summation
$$\sum_{d > z} \frac{|\ell_q(d)|}{\sqrt{d}} \ll_{q,\varepsilon} z^{-3/10 + \varepsilon}.$$
We infer that
$$\sum_{n \leqslant x} \left( \lambda_q \star \mathbf{1} \right) (n) \ll x^{1/2} \delta_c \left( x^{1/4} \right) + x^{7/20+\varepsilon} \ll x^{1/2} \delta_c \left( x^{1/4} \right).$$
Now suppose that the Riemann hypothesis is true. By \cite{bal}, which is a refinement of \cite{sou}, we know that $M(z) \ll_\varepsilon z^{1/2} \, \omega(z)$. The method of \cite{sou,bal} may be adapted to the function $L$ yielding
$$L(z) \ll_\varepsilon z^{1/2} \, \omega(z) \, \log z.$$
Observe that, for any $a \geqslant 2$, $\varepsilon>0$ and $z \geqslant e^{e^e}$
$$\log z \exp \left( \sqrt{\log z} \, (\log \log z)^a \right) \leqslant \exp \left( \sqrt{\log z} \, (\log \log z)^{a+\varepsilon} \right)$$
so that $L(z) \ll_\varepsilon z^{1/2} \, \omega(z)$ and hence
$$\sum_{n \leqslant x} \left( \lambda_q \star \mathbf{1} \right) (n) \ll x^{1/4} \sum_{d \leqslant x} \frac{|\ell_q(d)|}{d^{1/4}} \, \omega \left(  \sqrt{\frac{x}{d}} \right) \ll x^{1/4} \omega \left( \sqrt{x} \right) \sum_{d \leqslant x} \frac{|\ell_q(d)|}{d^{1/4}} \ll x^{1/4} \omega \left( \sqrt{x} \right)$$
achieving the proof in that case. The case $q=5$ is similar but simpler since $\lambda_5 \star \mathbf{1} = a_5 \star a_2^{-1}$ by \eqref{e3}.

Finally, when $q \equiv \pm 43, \pm 53 \; (\md 120)$, we proceed as above. Let $\nu_q(n)$ be the $n$-th coefficient of the Dirichlet series $\mathcal{L}_q(s)$. Then $\lambda_q \star \mathbf{1} = \nu_q \star a_4^{-1} \star a_2^{-1}$ from Lemma~\ref{le3}, so that
$$\sum_{n \leqslant x} \left( \lambda_q \star \mathbf{1} \right) (n) = \sum_{d \leqslant x} \nu_q(d) \sum_{m \leqslant (x/d)^{1/4}} \mu(m) M \left( \frac{1}{m^2}\sqrt{\frac{x}{d}} \right)$$
and estimating trivially yields
$$\sum_{n \leqslant x} \left( \lambda_q \star \mathbf{1} \right) (n) \ll x^{1/2} \sum_{d \leqslant x} \frac{|\nu_q(d)|}{\sqrt{d}} \sum_{m \leqslant (x/d)^{1/4}} \frac{1}{m^2} \, \delta_c \left( \frac{1}{m^2} \sqrt{\frac{x}{d}} \right) $$
and we complete the proof as in the previous case.
\qed

\begin{remark}
Let us stress that a bound of the shape
$$\sum_{n \leqslant x} \left( \lambda_q \star \mathbf{1} \right) (n) \ll x^{1/4 + \varepsilon}$$
for all $x$ sufficiently large and small $\varepsilon > 0$, is a necessary and sufficient condition for the Riemann hypothesis. Indeed, if this estimate holds, then by partial summation the series $\sum_{n=1}^\infty \left( \lambda_q \star \mathbf{1} \right) (n) n^{-s}$ is absolutely convergent in the half-plane $\sigma > \frac{1}{4}$. Consequently, the function $K_q(s) \zeta(2s)^{-1}$ is analytic in this half-plane. In particular, $\zeta(2s)$ does not vanish in this half-plane, implying the Riemann hypothesis, proving the necessary condition, the sufficiency being established above.
\end{remark}

\section{A short interval result for the case $q=5$}

\subsection{Introduction}

This section deals with sums of the shape
$$\sum_{x < n \leqslant x+y} \left( \lambda_5 \star \mathbf{1} \right) (n)$$
where $x^\varepsilon \leqslant y \leqslant x$. From Theorem~\ref{t1}
$$\sum_{x < n \leqslant x+y} \left( \lambda_5 \star \mathbf{1} \right) (n) \ll x^{1/2} e^{-c \left( \log x^{1/4}\right)^{3/5} \left( \log \log x^{1/4} \right)^{-1/5}}$$
and if the Riemann hypothesis is true, then 
$$\sum_{x < n \leqslant x+y} \left( \lambda_5 \star \mathbf{1} \right) (n) \ll_\varepsilon x^{1/4} e^{\left( \log \sqrt x \right)^{1/2} (\log \log \sqrt x)^{5/2+\varepsilon}}.$$
The purpose is to improve significantly upon these estimates when $y=o(x)$, by using fine results belonging to the theory of integer points near a suitably chosen smooth curve. To this end, we need the following additional specific notation. Let $\delta \in \left( 0,\frac{1}{4} \right)$, $N \in \Z_{\geqslant 1}$ large, $f : \left[ N,2N \right] \longrightarrow \R$ be any map, and define $\mathcal{R} (f,N,\delta)$ to be the number of elements of the set of integers $n \in \left[ N,2N \right]$ such that $\| f(n) \| < \delta$, where $\| x \|$ is the distance from $x$ to its nearest integer. Note that the trivial bound is given by
$$\sum_{x < n \leqslant x+y} \left( \lambda_5 \star \mathbf{1} \right) (n) \ll \sum_{x < n \leqslant x+y} \tau(n) \ll y \log x.$$

\subsection{Tools from the theory}

\noindent
In what follows, $N \in \Z_{\geqslant 1}$ is large and $\delta \in \left( 0,\frac{1}{4} \right)$. The first result is \cite[Theorem 5]{huxs} with $k=5$. See also \cite[Theorem 5.23 (iv)]{bor}.

\begin{lemma}[$5$th derivative test]
\label{le4}
Let $f \in C^5 \left[ N,2N \right]$ such that there exist $\lambda_4 > 0$ and $\lambda_5 > 0$ satisfying $\lambda_4 = N \lambda_5$ and, for any $x \in \left[ N,2N \right]$
$$\left | f^{(4)} (x) \right | \asymp \lambda_4 \quad \text{and} \quad \left | f^{(5)} (x) \right | \asymp \lambda_5.$$
Then
$$\mathcal{R} (f,N,\delta) \ll N \lambda_5^{1/15} + N \delta^{1/6} + \left( \delta \lambda_4^{-1} \right)^{1/4}+1.$$
\end{lemma}

\begin{remark}
The basic result of the theory is the following first derivative test (see \cite[Theorem 5.6]{bor}): \textit{Let $f \in C^1 \left[ N,2N \right]$ such that there exist $\lambda_1 > 0$ such that $\left | f^{\, \prime} (x) \right | \asymp \lambda_1$. Then}
\begin{equation}
   \mathcal{R} (f,N,\delta) \ll N \lambda_1 + N \delta + \delta \lambda_1^{-1}+1. \label{e4}
\end{equation}
This result is essentially a consequence of the mean value theorem.
\end{remark}

The second tool is \cite[Theorem 7]{fil} with $k=3$.

\begin{lemma}
\label{le5}
Let $s \in \Q^* \setminus \{ \pm 2, \pm 1 \}$ and $X > 0$ such that $N \leqslant X^{1/s}$. Then there exists a constant $c_3 := c_3(s) \in \left( 0,\frac{1}{4} \right)$ depending only on $s$ such that, if
\begin{equation}
   N^2 \delta \leqslant c_3 \label{e5}
\end{equation}
then
$$\mathcal{R} \left (\frac{X}{n^s},N,\delta \right ) \ll \left( XN^{3-s} \right)^{1/7} + \delta \left( XN^{59-s} \right)^{1/21}.$$
\end{lemma}

Our last result relies the short sum of $\lambda_5 \star \mathbf{1}$ to a problem of counting integer points near a smooth curve.

\begin{lemma}
\label{le6}
Let $1 \leqslant y \leqslant x$. Then
$$\sum_{x < n \leqslant x+y} \left( \lambda_5 \star \mathbf{1} \right) (n) \ll \underset{\left( 16y^2x^{-1} \right)^{1/5} <N \leqslant \left( 2x \right)^{1/5}}{\max} \mathcal{R} \left( \sqrt{\frac{x}{n^5}},N,\frac{y}{\sqrt{N^5x}}\right) \log x + yx^{-1/2} + x^{-1/5} y^{2/5}.$$
\end{lemma}

\begin{proof}
Using \eqref{e3}, we get
$$\sum_{n \leqslant x} \left( \lambda_5 \star \mathbf{1} \right) (n) = \sum_{d \leqslant \sqrt{x}} \mu(d) \left \lfloor \left( \frac{x}{d^2} \right)^{1/5} \right \rfloor$$
so that
\begin{eqnarray*}
   \sum_{x < n \leqslant x+y} \left( \lambda_5 \star \mathbf{1} \right) (n) &=& \sum_{d \leqslant \sqrt{x}} \mu(d) \left( \left \lfloor \left( \frac{x+y}{d^2} \right)^{1/5} \right \rfloor - \left \lfloor \left( \frac{x}{d^2} \right)^{1/5} \right \rfloor \right)  + \sum_{\sqrt{x} < d \leqslant \sqrt{x+y}} \mu(d) \\
   & \ll & \sum_{d \leqslant \sqrt{x}} \left( \left \lfloor \left( \frac{x+y}{d^2} \right)^{1/5} \right \rfloor - \left \lfloor \left( \frac{x}{d^2} \right)^{1/5} \right \rfloor \right) + yx^{-1/2} \\
   & \ll & \sum_{d \leqslant \sqrt{x}} \; \sum_{x < d^2 n^5 \leqslant x+y} 1 + yx^{-1/2} \\
   & \ll & \sum_{n \leqslant (2x)^{1/5}} \sum_{\left( \frac{x}{n^5} \right)^{1/2} < d \leqslant \left( \frac{x+y}{n^5} \right)^{1/2}} 1 + yx^{-1/2} \\
   & \ll &  \sum_{\left (16y^2x^{-1} \right)^{1/5} < n \leqslant \left( 2x\right)^{1/5}} \left( \left \lfloor \sqrt{\frac{x+y}{n^5}} \right \rfloor - \left \lfloor \sqrt{\frac{x}{n^5}} \right \rfloor \right) + x^{-1/5} y^{2/5} + yx^{-1/2}
\end{eqnarray*}
and for any integers $N \in \left] \left (16y^2x^{-1} \right)^{1/5}, \left( 2x \right)^{1/5} \right]$ and $n \in \left[ N,2N \right]$
$$\sqrt{\frac{x+y}{n^5}} - \sqrt{\frac{x}{n^5}} < \frac{y}{\sqrt{N^5x}} < \frac{1}{4}$$
so that the sum does not exceed
$$\ll \underset{\left( 16y^2x^{-1} \right)^{1/5} <N \leqslant \left( 2x \right)^{1/5}}{\max} \mathcal{R} \left( \sqrt{\frac{x}{n^5}},N,\frac{y}{\sqrt{N^5x}}\right) \log x+ x^{-1/5} y^{2/5} + yx^{-1/2}$$
as asserted.
\end{proof}

\subsection{The main result}

\begin{theorem}
\label{t6}
Assume $y \leqslant c_3 \, x^{11/20}$ where $c_3 := c_3 \left( \frac{5}{2} \right)$ is given in \eqref{e5}. Then
$$\sum_{x < n \leqslant x+y} \left( \lambda_5 \star \mathbf{1} \right) (n) \ll \left( x^{1/12} + yx^{-4/9} \right) \log x.$$
Furthermore, if $y \leqslant c_3 \, x^{19/36}$
$$\sum_{x < n \leqslant x+y} \left( \lambda_5 \star \mathbf{1} \right) (n) \ll x^{1/12} \log x.$$
\end{theorem}

\begin{proof}
We split the first term in Lemma~\ref{le6} into three parts, according to the ranges
$$\left( 16y^2x^{-1} \right)^{1/5} < N \leqslant 2x^{1/10}, \quad 2x^{1/10} < N \leqslant 2x^{1/6} \quad \textrm{and} \quad 2x^{1/6} < N \leqslant (2x)^{1/5}.$$
In the first case, we use Lemma~\ref{le4} with $\lambda_4 = \left( x N^{-13} \right)^{1/2}$ and $\lambda_5 = \left( x N^{-15} \right)^{1/2}$ which yields
$$\underset{\left( 16y^2x^{-1} \right)^{1/5} <N \leqslant 2x^{1/10}}{\max} \ \mathcal{R} \left( \sqrt{\frac{x}{n^5}},N,\frac{y}{\sqrt{N^5x}}\right) \ll x^{1/12} + x^{-1/40} y^{1/6} + x^{-3/20} y^{1/4}.$$
For the second range, we use Lemma~\ref{le5} with $X=x^{1/2}$, $s = \frac{5}{2}$ and $\delta = y \left( N^5 x \right)^{-1/2}$. Notice that the conditions $N > 2x^{1/10}$ and $y \leqslant c_3 \, x^{11/20}$ ensure that $\delta < \frac{1}{4}$ and $N^2 \delta \leqslant c_3$. We get 
$$\underset{2x^{1/10} < N \leqslant 2x^{1/6}}{\max} \ \mathcal{R} \left( \sqrt{\frac{x}{n^5}},N,\frac{y}{\sqrt{N^5x}}\right) \ll x^{1/12} + yx^{-4/9}.$$
The last range is easily treated with \eqref{e4}, giving
$$\underset{2x^{1/6} < N \leqslant \left( 2x \right)^{1/5}}{\max} \ \mathcal{R} \left( \sqrt{\frac{x}{n^5}},N,\frac{y}{\sqrt{N^5x}}\right) \ll x^{1/12} + yx^{-3/4}.$$
Using Lemma~\ref{le6}, we finally get
$$\sum_{x < n \leqslant x+y} \left( \lambda_5 \star \mathbf{1} \right) (n) \ll \left( x^{1/12} + x^{-1/40} y^{1/6} + x^{-3/20} y^{1/4} + yx^{-4/9} \right) \log x  + x^{-1/5} y^{2/5}$$
and note that $x^{-1/40} y^{1/6} + x^{-3/20} y^{1/4} + x^{-1/5} y^{2/5} \ll x^{1/12}$ as soon as $y \leqslant x^{13/20}$. This completes the proof of the first estimate, the second one being obvious.
\end{proof}

\section{Acknowledgments}

The author deeply thanks Prof. Kannan Soundararajan for the help he gave him to adapt his result to the function $L(x)$, and Benoit Cloitre for bringing this problem to his attention.


\begin{thebibliography}{9}
   \bibitem{bal} M. Balazard and A. de Roton, Notes de lecture de l'article "Partial sums of the M\"{o}bius function" de Kannan Soundararajan. \textit{arXiv.org}, 2008, arXiv:0810.3587v1.
   \bibitem{bor} O. Bordell\`{e}s, \textit{Arithmetic Tales}, Springer, 2012. 
   \bibitem{dav} H. Davenport, \textit{The Higher Arithmetic}, 5th edition, Cambridge University Press, London, New York, 1982.
   \bibitem{fil} M. Filaseta and O. Trifonov, The distribution of fractional parts with applications to gap results in number theory, \textit{Proc. London Math. Soc.} \textbf{73}(3) (1996), 241--278.
   \bibitem{har} G. H. Hardy, On Dirichlet's divisor problem, \textit{Proc. London Math. Soc.} \textbf{15} (1916), 1--25.
   \bibitem{hux} M. N. Huxley, Exponential sums and lattice points III, \textit{Proc. London Math. Soc.} \textbf{87} (2003), 591--609.
   \bibitem{huxs} \textsc{M. N. Huxley \& P. Sargos}, Points entiers au voisinage d'une courbe plane de classe $C^n$, II, \textit{Functiones et Approximatio} \textbf{35} (2006), 91--115.
   \bibitem{mut} R. K. Muthumalai, Note on Legendre symbols connecting with certain infinite series, \textit{Notes on Number Theory and Discrete Mathematics} \textbf{19} (2013), 77--83.
   \bibitem{sou} K. Soundararajan, Partial sums of the M\"{o}bius function, \textit{J. Reine Angew. Math.} \textbf{631} (2009), 141--152.
\end{thebibliography}
\end{document}